\documentclass[10pt,reqno]{amsart}

\usepackage{amssymb}
\usepackage{amsfonts}
\usepackage{amsthm}
\usepackage{amsmath}
\usepackage{bbm}

\numberwithin{equation}{section}
\allowdisplaybreaks

\newtheorem{theorem}{Theorem}[section]
\newtheorem{proposition}[theorem]{Proposition}

\newtheorem{lemma}[theorem]{Lemma}
\newtheorem{corollary}[theorem]{Corollary}

\theoremstyle{definition}

\theoremstyle{remark}

\newcommand{\R}{\mathbb{R}}

\newcommand{\Z}{\mathbb{Z}}

\newcommand{\eps}{\varepsilon}

\newcommand{\scriptD}{\mathcal{D}}
\newcommand{\scriptE}{\mathcal{E}}

\newcommand{\scriptH}{\mathcal{H}}
\newcommand{\scriptI}{\mathcal{I}}

\newcommand{\scriptK}{\mathcal{K}}

\newcommand{\scriptT}{\mathcal{T}}

\begin{document}
\title{Scale invariant Fourier restriction to a hyperbolic surface}
\author{Betsy Stovall}
\address{University of Wisconsin--Madison, 480 Lincoln Drive, Madison, Wisconsin, 53706}
\email{stovall@math.wisc.edu}
\begin{abstract}
This result sharpens the bilinear to linear deduction of Lee and Vargas for extension estimates on the hyperbolic paraboloid in $\R^3$ to the sharp line, leading to the first scale-invariant restriction estimates, beyond the Stein--Tomas range, for a hypersurface on which the principal curvatures have different signs.  
\end{abstract}
\maketitle

\section{Introduction}

We consider the Fourier restriction/extension problem for the hyperbolic paraboloid 
$$
S:= \{(\tau,\xi) \in \R^{1+2} : \tau = \xi_1 \xi_2\}.
$$
We denote by $\scriptE$ the extension operator,
\begin{equation} \label{E:def extn}
\scriptE f(t,x) := \int_{\R^2} e^{i(t,x)(\xi_1\xi_2,\xi)} f(\xi)\, d\xi.
\end{equation}
For consistency of exponents, we will consider the problem of establishing $L^r \to L^{2s}$ extension estimates for $\scriptE$, and we are primarily interested in the case when $r=s'$.  

In \cite{LeeNeg, VargasNeg}, Lee and Vargas independently established an essentially optimal $L^2$-based bilinear adjoint restriction estimate for $S$.  This result states that if $f$ and $g$ are supported in $1 \times 1$ rectangles that are separated from one another by a distance 1 in the horizontal direction and 1 in the vertical direction, then
\begin{equation} \label{E:L2 bilin}
\|\scriptE f \scriptE g\|_s \lesssim \|f\|_2\|g\|_2, \qquad s > \tfrac53.
\end{equation}
This two-parameter separation of the tiles is both necessary and troublesome.  On the one hand, necessity can be seen by considering the case when each of $f_\pm$ is supported on a $\tfrac12$-neighborhood of $(\pm 1,0)$.  On the other hand, the separation leads to difficulty in deducing linear restriction estimates from the bilinear ones.  Indeed, the natural analogue of the Whitney decomposition approach of \cite{TVV} leads to a sum in two scales, length and width, rather than a single distance scale, leading to a loss of the scaling line in the distinct approaches of Lee \cite{LeeNeg} and Vargas \cite{VargasNeg}.

The purpose of this note is to overcome this obstacle and recover the sharp line.  

\begin{theorem} \label{T:main}
With $\scriptE$ as in \eqref{E:def extn}, assume that the estimate
\begin{equation} \label{E:LsLr bilin 0}
\|\scriptE f \scriptE g\|_s \lesssim \|f\|_r\|g\|_r
\end{equation}
holds for some $\tfrac 32 < s < 2$ and $r < s'$, whenever $f$ and $g$ are supported on $1 \times 1$ rectangles that are separated from one another by a distance 1 in both the horizontal and vertical directions.  Then $\scriptE$ is of restricted strong type $(s',2s)$, and consequently of strong type $(\tilde s',2\tilde s)$ for all $ \tilde s > s$.  
\end{theorem}

To put the hypothesis on $s$ in context, we recall that for $s \leq \tfrac32$, linear extension estimates are known to be impossible, and for $s \geq 2$, they are already known, \cite{Tomas}.  

As is well-known, a (local, linear) $L^{r_0} \to L^{2s_0}$ extension estimate for some $r_0 > s_0'$ allows us, by interpolation with the $L^2$-based bilinear extension estimate \eqref{E:L2 bilin}, to establish the $L^r$-based bilinear extension estimate \eqref{E:LsLr bilin 0} for some $s > s_0$ and $r < s'$.  Replacing $s_0$ with $s$ is a loss (the extent of which depends on $1-r_0^{-1}-s_0^{-1}$), but $r < s'$ is a gain in the sense that the corresponding linear extension estimate $\scriptE:L^r \to L^{2s}$ is false.  

In \cite{LeeNeg, VargasNeg}, Lee and Vargas independently used the bilinear extension estimate \eqref{E:L2 bilin} to prove that 
\begin{equation} \label{E:extn}
\|\scriptE f\|_{2s} \lesssim \|f\|_{L^r}, 
\end{equation}
for all $s > \tfrac{5}3$, $r > s'$, and $f$ supported in the unit ball.  In \cite{ChoLee}, Cho--Lee used Guth's polynomial partitioning argument from \cite{Guth} to prove \eqref{E:extn} for $f$ supported in the unit ball and $2s=r > 3.25$; this was subsequently improved by Kim \cite{Jongchon} to the range $2s > 3.25$ and $r > s'$.  Using these results and the discussion in the preceding paragraph, Theorem~\ref{T:main} immediately yields the following slight improvement on Kim's result.

\begin{corollary} \label{C:linear restriction range}
For $2s > 3.25$, the extension operator $\scriptE$ is bounded from $L^{s'}$ to $L^{2s}$.  
\end{corollary}

To the author's knowledge, this is the first scalable restriction estimate for a negatively curved hypersurface, beyond the Stein--Tomas range ($s=2$).  

\subsection*{Terminology} A constant will be said to be admissible if it depends only on $s,r$.  The inequality $A \lesssim B$ means that $A \leq CB$ for some implicit constant $C$, and implicit constants will be allowed to change from line to line.  A dyadic interval is an interval of the form $[m2^{-n},(m+1)2^{-n}]$, for some $m,n \in \Z$, and $\scriptI_n$ denotes the set of all dyadic intervals of length $2^{-n}$. A tile is a product of two dyadic intervals, and $\scriptD_{J,K}$ denotes the set of all $2^{-J} \times 2^{-K}$ tiles.  

\subsection*{Outline of proof}  To prove our restricted strong type estimate, it suffices to bound the extension of a characteristic function.  Our starting point is the bilinear to linear deduction of Vargas \cite{VargasNeg}, which shows that, under the hypotheses of Theorem~\ref{T:main}, the extension of the characteristic function of a set $\Omega$ with roughly constant fiber length obeys the scalable restriction estimate $\|\scriptE \chi_\Omega\|_{2s} \lesssim |\Omega|^{\frac1{s'}}$.  In \cite{VargasNeg}, off-scaling estimates are obtained by subdividing a set $\Omega$ in the unit cube into subsets having constant fiber length.  Off-scaling contributions from those subsets with very short fibers are small (because the sets themselves are small), and adding these amounts to summing a geometric series.  

We wish to remain on the sharp line, so must be more careful.  Our first step, taken in Section~2, is to understand when Vargas's constant fiber estimate can be improved.  To this end, we prove a dichotomy result:  If $\Omega$ has constant fiber length, then either $\Omega$ is highly structured (more precisely, $\Omega$ is nearly a tile), or we have a better bound on the extension of $\chi_\Omega$.  Roughly speaking, this reduces matters to controlling the extension of a union of tiles $\tau_k$ having heights $2^{-k}$, which is the task of Section~3.  We can  estimate
$$
\|\scriptE \chi_{\bigcup \tau_k}\|_{2s} \lesssim \bigl(\sum \|\scriptE \chi_{\tau_k}\|_{2s}^{2s}\bigr)^{\frac1{2s}} + \text{off-diagonal terms},
$$
where the off-diagonal terms involve products $\scriptE \chi_{\tau_k}\scriptE \chi_{\tau_{k'}}$, with $|k-k'|$ large.    
Boundedness of the main term follows from Vargas's estimate and convexity ($2s>s'$).  It remains to bound the off-diagonal terms, for which it suffices to prove a bilinear estimate with decay:  
$$
\|\scriptE \chi_{\tau_k} \scriptE \chi_{\tau_{k'}}\|_s \lesssim 2^{-c_0|k-k'|}\max\{|\tau_k|,|\tau_{k'}|\}^{\frac1{s'}},
$$
and we prove this by combining the bilinear extension estimate for separated tiles with a further decomposition.  

Of course, we have lied.  In Section~2, our dichotomy is not that a constant fiber length set $\Omega$ is either a tile or has zero extension, and so we still have remainder terms that must be summed.  To address this, we argue more quantitatively than has been suggested above:  Any constant fiber length set can be approximated by a union of tiles, where the number of tiles and tightness of the approximation depends on how sharp is our estimate $\|\scriptE\chi_\Omega\|_{2s}\lesssim |\Omega|^{\frac1{s'}}$; then we must bound extensions of sets $\bigcup_k \bigcup_{\tau \in \scriptT_k}\tau$, where $\scriptT_k \subseteq \scriptD_{j(k),k}$ may be large (but fortunately, not too large).  

\subsection*{Acknowledgements}  This work has been supported in part by a grant from the National Science Foundation (DMS-1600458), and was carried out in part while the author was in residence at the Mathematical Sciences Research Institute (MSRI) in Berkeley, California, during the Spring of 2017, a visit that was supported in part by a National Science Foundation grant to MSRI (DMS-1440140).  The author would like to thank Sanghyuk Lee, Andreas Seeger, and Ana Vargas, from whom she learned of this problem and some of its history.

\section{An inverse problem related to Vargas's linear estimate}

To prove Theorem~\ref{T:main}, it suffices to prove that $\|\scriptE \chi_\Omega\|_{2s} \lesssim |\Omega|^{\frac1{s'}}$, for all measurable sets $\Omega$.  By scaling, it suffices to consider $\Omega$ contained in the unit cube, $\tau_0$.  In \cite{VargasNeg}, Vargas proved the following.

\begin{theorem}[Vargas, \cite{VargasNeg}] \label{T:Vargas}
For each $K \geq 0$, let 
$$
\Omega(K) := \{\xi \in \Omega : \scriptH^1(\pi_1^{-1}(\xi_1)) \sim 2^{-K}\}.
$$
Then under the hypotheses of Theorem~\ref{T:main}, for any measurable set $\Omega' \subseteq \Omega(K)$,
\begin{equation} \label{E:OmegaK}
\|\scriptE \chi_{\Omega'}\|_{2s} \lesssim |\Omega(K)|^{\frac1{s'}}.
\end{equation}
\end{theorem}

This version differs slightly from the one stated in \cite{VargasNeg}, but it follows from the same proof.  In proving the next proposition, we will review Vargas's argument, so the reader may verify the above-stated version below.

Our first step is to solve an inverse problem:  Characterize those sets $\Omega = \Omega(K)$ for which the inequality in \eqref{E:OmegaK} can be reversed.  We record here a useful rescaling of the bilinear estimate \eqref{E:LsLr bilin 0}, namely, for $f,g$ supported on tiles in $\scriptD_{j,k}$ that are separated by a distance $2^{-k}$ in the vertical direction and $2^{-j}$ in the horizontal direction, \begin{equation} \label{E:LsLr bilin}
\|\scriptE f \scriptE g\|_{ s} \lesssim 2^{-(j+k)(2-\tfrac2{ s} - \tfrac 2{ r})}\|f\|_{ r} \|g\|_{ r}.
\end{equation}

\begin{proposition}\label{P:Omega tile}
Assume that the hypotheses of Theorem~\ref{T:main} hold.  Let $\Omega \subseteq \tau_0$ be a measurable set, and assume that $\Omega = \Omega(K)$ for some integer $K \geq 0$.  Choose a nonnegative integer $J$ such that $|\pi_1(\Omega)| \sim 2^{-J}$, and let $\eps \lesssim 1$ denote the smallest dyadic number such that
$$
\|\scriptE \chi_{\Omega'}\|_{2s} \leq \eps |\Omega|^{\frac1{s'}}, 
$$
for every measurable $\Omega' \subseteq \Omega$.  Then $\Omega = \bigcup_{0 < \delta \leq \eps} \Omega_\delta$, with the union taken over dyadic $\delta$.  For each $\delta$, $\Omega_\delta$ is contained in a union of $O(\delta^{-C})$ tiles $\tau \in \scriptT_\delta \subseteq \scriptD_{J,K}$, and for each subset $\Omega' \subseteq \Omega_\delta$, $\|\scriptE \chi_{\Omega'}\|_{2s}\lesssim \delta |\Omega|^{\frac1{s'}}$.  
\end{proposition}

\begin{proof}[Proof of Proposition~\ref{P:Omega tile}]
Our decomposition will be done in three stages.  Our first decomposition will be of $\Omega$ into sets $\Omega_\eta^1$, with $\pi_1(\Omega_\eta^1)$ nearly an interval, $I \in \scriptI_J$.  Our second decomposition will be of $\Omega_\eta^1$ into sets $\Omega_{\eta,\rho}^2$, $\rho \leq \eta$, each of which is nearly a product of $I$ with a set of measure $2^{-K}$.  Our third decomposition will be of $\Omega_{\eta,\rho}^2$ into sets $\Omega_{\eta, \rho, \delta}^3$, $\delta \leq \rho$, each of which is nearly a product of $I$ with an interval in $\scriptI_K$.  The product of two dyadic intervals is a tile, so we take $\Omega_\delta := \bigcup_{\rho \geq \delta}\bigcup_{\eta \geq \rho} \Omega_{\eta,\rho,\delta}^3$; the $(\log \delta^{-1})^2$ factor that arises from taking this union is harmless.

Let $S := \pi_1(\Omega)$.  We know that $|S| \sim 2^{-J}$ and that $S \subseteq [-1,1]$.  Let $\xi_1\in S$, and for each $0 < \eta < \eps$, let $I_\eta(\xi_1)$ be the maximal dyadic interval $I \ni \xi_1$ satisfying $|I \cap S| \geq \eta^C |I|$, if such an interval exists.  We record that $|I_\eta(\xi_1)| \leq \eta^{-C}2^{-J}$, and if $\xi_1$ is a Lebesgue point of $S$, then $|I_\eta(\xi_1)| > 0$.  Let 
$$
T_\eta := \{\xi_1 \in S : |I_\eta(\xi_1)| \geq \eta^C 2^{-J}\},
$$
and let $S_\eps := T_\eps$, $S_\eta := T_\eta \setminus T_{2\eta}$, for dyadic $0 < \eta < \eps$.  Then a.e.\ (indeed, every Lebesgue) point of $S$ is contained in a unique $S_\eta$.  We set $\Omega_\eta^1 := \Omega \cap \pi_1^{-1}(S_\eta)$.  

\begin{lemma} \label{L:Omega eta}
For each $0 < \eta \leq \eps$, $S_\eta$ is contained in a union of $O(\eta^{-2C})$ dyadic intervals $I \in \scriptI_J$, and for each  $\eta < \eps$ and each subset $\Omega' \subseteq \Omega_\eta^1$, $\|\scriptE \chi_{\Omega'}\|_{2s}\lesssim \eta^2 |\Omega|^{\frac1{s'}}$.  
\end{lemma}

\begin{proof}[Proof of Lemma~\ref{L:Omega eta}]
All of the conclusions except for the bound on the extension $\scriptE \chi_{\Omega'}$ with $\Omega' \subseteq \Omega_\eta^1$ are immediate.  To establish this bound, we optimize Vargas's proof of Theorem~\ref{T:Vargas}.  The argument is largely the same as that in \cite{VargasNeg}, so we will be brief.  Performing a Whitney decomposition in each variable $\xi_1,\xi_2$ separately and applying the almost orthogonality lemma from \cite{TVV} (for which it is important that $s\leq 2$),
$$
\|\scriptE \chi_{\Omega'}\|_{2s}^2 \leq \sum_{k,j} \bigl(\sum_{\tau \sim \tau' \in \scriptD_{j,k}} \|\scriptE \chi_{\Omega' \cap \tau} \scriptE \chi_{\Omega' \cap \tau}\|_s^s\bigr)^{\frac1s},
$$
where we say that $\tau \sim \tau'$ if $\tau$ and $\tau'$ are $2^{-j}$ separated in the horizontal direction and $2^{-k}$ separated in the vertical direction.  

From our hypothesis that we have bilinear extension estimates for some $r < s' < 2s$,
\begin{align*}
\|\scriptE \chi_{\Omega'}\|_{2s}^2 &\lesssim \sum_{k,j} 2^{-(j+k)(2-\frac 2s - \frac 2r)}\bigl(\sum_\tau |\Omega' \cap \tau|^{\frac{2s}r}\bigr)^{\frac1s} \\
&\lesssim \sum_{k,j} 2^{-(j+k)(2-\frac 2s - \frac 2r)} \max_{\tau \in \scriptD_{j,k}} |\Omega' \cap \tau|^{\frac2r-\frac1s}|\Omega'|^{\frac1s}.
\end{align*}
To bound this double sum, Vargas used the inequality
\begin{equation} \label{E:Vargas size bound}
|\Omega' \cap \tau| \lesssim \min\{2^{-j},2^{-J}\}\min\{2^{-k},2^{-K}\}.
\end{equation}
The definition of $\Omega_\eta^1$ will allow us to improve on this bound.  

Take $C$ exactly as in the definition of $T_\eta$.  For $I_j \in \scriptI_j$, we have the trivial bound
$$
|I_j \cap S_\eta| \leq \min\{|I_j|,|S_\eta|\} \leq \min\{2^{-j},2^{-J}\},
$$
but when $|j-J| < C\log \eta^{-1}$, we get a dramatic improvement.  Indeed, if $\eta^C2^{-J} \leq 2^{-j} \leq \eta^{-C}2^{-J}$, then $|S_\eta \cap I_j| < \eta^C|I_j|$, and if $\eta^{2C}2^{-J} < 2^{-j} < \eta^C2^{-J}$, then $|S_\eta \cap I_j| \lesssim \eta^{2C}2^{-J}$; in each case, were the stated bound to fail, we could find a point $\xi_1 \in S_\eta$ that belonged to some $S_{\eta'}$ with $\eta' > \eta$, a contradiction.  As 
$$
|\Omega' \cap (I_j \times I_k)| \leq |S_\eta \cap I_j| \min\{2^{-k},2^{-K}\},
$$
the above improvement and the inequalities $r < s' < 2s$ lead to 
$$
\|\scriptE \chi_{\Omega'}\|_{2s} \lesssim \eta^{C'} 2^{-(J+K)(\frac1r-\frac1{2s})}|\Omega'|^{\frac1{2s}} \lesssim \eta^{C'}|\Omega|^{\frac1{s'}},
$$
for $C'>0$ some admissible constant dictated by $C,r,s$; we can reverse engineer $C$ so that $C'=2$.  
\end{proof}

Now for our second decomposition.  Although $\pi_1(\Omega_\eta^1)$ (roughly) projects down to a small number of intervals, an individual horizontal slice $\pi_2^{-1}(\xi_2) \cap \Omega_\eta^1$ might be much smaller.  Our next step is to decompose into sets where the size of a nonempty slice is roughly comparable to the size of the projection of the whole.  (Sets with this property are nearly products.)  

Fix $0 < \eta \leq \eps$.  For dyadic $0 < \rho \leq \eta$, we define
$$
V_\rho = \{\xi_2 \in \pi_2(\Omega_\eta^1) : \scriptH^1(\pi_2^{-1}(\xi_2) \cap \Omega_\eta^1) \geq \rho^C2^{-J}\},
$$
and set $U_\eta := V_\eta$, $U_\rho := V_\rho \setminus V_{2\rho}$, for $\rho < \eta$.  We define $\Omega^2_{\eta,\rho} := \pi_2^{-1}(U_\rho) \cap \Omega_\eta^1$.  

\begin{lemma} \label{L:Omega eta rho}
For each $0 < \rho < \eta \leq \eps$, and each subset $\Omega' \subseteq \Omega_{\eta,\rho}^2$, $\|\scriptE \chi_{\Omega'}\|_{2s}\lesssim \rho^2 |\Omega|^{\frac1{s'}}$.
\end{lemma}

\begin{proof}[Proof of Lemma~\ref{L:Omega eta rho}]
The proof is similar to Lemma~\ref{L:Omega eta}, the only difference being that the bound on the intersection of a tile $\tau \in \scriptD_{j,k}$ with $\Omega' \subseteq \Omega_{\eta,\rho}^2$ is
$$
|\tau \cap \Omega'| \lesssim \min\{\rho^C 2^{-J},2^{-j}\} \min\{2^{-K},2^{-k}\}.
$$
Alternately, one may deduce this directly from Vargas' Theorem~\ref{T:Vargas} by interchanging the indices, and then using the size estimate $|\Omega_{\eta,\rho}^2| \lesssim \rho^C 2^{-(J+K)} \sim \rho^C|\Omega|$, for $\rho < \eta$.  
\end{proof}

Now our third decomposition.  A single $\Omega_{\eta,\rho}^2$ is ``nearly'' a product, but $\pi_2(\Omega_{\eta,\rho}^2)$ might be far from an interval.  This can be fixed in a similar manner to the decomposition of $\Omega$ into the $\Omega_\eta^1$:  Indeed, we perform exactly the same decomposition as before, only interchanging the roles of the indices.  

We complete the proof by taking unions as described at the outset.  The factors of $\eta^2$ and $\rho^2$ in Lemmas~\ref{L:Omega eta} and~\ref{L:Omega eta rho} (and the factor of $\delta^2$ in the analogue for $\Omega_{\eta,\rho,\delta}^3$) mean that the resulting factor of $(\log \delta^{-1})^2$ is indeed harmless.  
\end{proof}

\section{Extensions of characteristic functions of near tiles}

We will complete the proof of Theorem~\ref{T:main} by summing the extensions of the sets that arise in Proposition~\ref{P:Omega tile}.  Let $\scriptK(\eps)$ denote the collection of all $K \in \Z_{\geq 0}$ for which $\eps$ is the smallest dyadic number for which $\|\scriptE \chi_{\Omega'}\|_{2s} \leq \eps |\Omega(K)|^{\frac1{s'}}$ holds for all $\Omega' \subseteq \Omega(K)$.  

\begin{lemma} \label{L:Omega delta decouple}
For $\eps > 0$, $0 < \delta \leq \eps$, under the hypotheses of Theorem~\ref{T:main},
$$
\|\sum_{K \in \scriptK(\eps)} \scriptE \chi_{\Omega_\delta(K)}\|_{2s}^{2s} \lesssim \bigl(\log \delta^{-1}\bigr)^2 \sum_{K \in \scriptK(\eps)} \|\scriptE \chi_{\Omega_\delta(K)}\|_{2s}^{2s} + \delta |\Omega|^{\frac{2s}{s'}}.
$$ 
\end{lemma}

\begin{proof}[Proof of Lemma~\ref{L:Omega delta decouple}]
It suffices to prove 
$$
\|\sum_{K \in \scriptK} \scriptE \chi_{\Omega_\delta(K)}\|_{2s}^{2s} \lesssim \sum_{K \in \scriptK} \|\scriptE \chi_{\Omega_\delta(K)}\|_{2s}^{2s} + \delta^2 |\Omega|^{\frac{2s}{s'}},
$$
with $\scriptK \subseteq \scriptK(\eps)$ chosen so that $\scriptK$ and $J(\scriptK)$ are both $A \log \delta^{-1}$-separated, with $A$ a sufficiently large admissible constant.  ($A$ will be much larger than the constant $C$ in Proposition~\ref{P:Omega tile}.)  Since $s < 2$, the triangle inequality gives
\begin{align*}
&\|\sum_{K \in \scriptK} \scriptE \chi_{\Omega_\delta(K)}\|_{2s}^{2s} = \int \bigl|\sum_{\mathbf K \in \scriptK^4} \prod_{i=1}^4 \scriptE \chi_{\Omega_\delta(K_i)}\bigr|^{\frac s2}\\
&\qquad\qquad \lesssim \sum_{K \in \scriptK} \|\scriptE \chi_{\Omega_\delta(K)}\|_{2s}^{2s} + {\sum}' \|\prod_{i=1}^4 \scriptE \chi_{\Omega_\delta(K_i)}\|_{\frac s2}^{\frac s2},
\end{align*}
where ${\sum}'$ indicates a sum taken on quadruples $\mathbf K = (K_1, K_2,K_3, K_4) \in \scriptK^4$, with at least two entries distinct.  We take a moment from the proof of Lemma~\ref{L:Omega delta decouple} to prove the following.

\begin{lemma} \label{L:separated tiles}
If $K,K' \in \scriptK$, and $J := J(K)$, $J' := J(K')$, then 
\begin{equation} \label{E:separated tiles}
\|\scriptE \chi_{\Omega_\delta(K)} \scriptE \chi_{\Omega_\delta(K')}\|_s \lesssim 2^{-c_0|K-K'|}\max\{|\Omega(K)|,|\Omega(K')|\}^{\frac2{s'}},
\end{equation}
for some admissible constant $c_0>0$.
\end{lemma}

\begin{proof}[Proof of Lemma~\ref{L:separated tiles}]
If $K=K'$, the inequality is a trivial consequence of Cauchy--Schwarz and \eqref{E:OmegaK}.  If $J=J'$, we again apply Cauchy--Schwarz and \eqref{E:OmegaK}, since
$$
|\Omega(K)|^{\frac1{s'}}|\Omega(K)|^{\frac1{s'}} \sim 2^{-\frac{|K-K'|}{s'}}\max\{|\Omega(K)|,|\Omega(K')|\}^{\frac2{s'}}.
$$
Thus it remains to consider the cases when $J$ and $J'$, and likewise, $K$ and $K'$, differ.  By symmetry, it suffices to consider the cases $J < J'$, $K < K'$; and $J > J'$, $K < K'$.  

If $J < J'$ and $K < K'$, then $|\Omega(K)| \sim 2^{-(J+K)} \geq 2^{-|K-K'|}|\Omega(K')|$, so \eqref{E:separated tiles} follows from Theorem~\ref{T:Vargas} and Cauchy--Schwarz.  

Thus we may assume that $K < K'$ and $J > J'$.  By Proposition~\ref{P:Omega tile} and the separation condition on $\scriptK$, it suffices to prove that 
\begin{equation} \label{E:log sep near tiles}
\|\scriptE \chi_{\Omega_\tau(K)} \scriptE \chi_{\Omega_{\tau'}(K')}\|_s \lesssim \delta^{-C} 2^{-c_0|K-K'|}|\Omega(K)|^{\frac1{s'}}|\Omega(K')|^{\frac1{s'}},
\end{equation}
for tiles $\tau \in \scriptT_\delta(K)$, $\tau' \in \scriptT_\delta(K')$.  

Note that our conditions on $J,J',K,K'$ mean that $\tau$ is taller than $\tau'$, and $\tau'$ is wider than $\tau$.  By translating, we may assume that the $y$-axis forms the center line of $\tau$ and that the $x$-axis forms the center line of $\tau'$.  Recalling that our tiles are contained in $2\tau_0$, we decompose:
\begin{gather*}
\tau = \bigcup_{k=0}^{K'} \tau_k, \qquad \tau_k = \tau \cap \{\xi : |\xi_2| \sim 2^{-k}\},\\
\tau' = \bigcup_{j=0}^J \tau_j', \qquad \tau_j' = \tau' \cap \{\xi : |\xi_1| \sim 2^{-j}\}.
\end{gather*}
By the (2-parameter) Littlewood--Paley square function estimate (the two-parameter version can be proved using Khintchine's inequality), the fact that $s < 2$, and the triangle inequality,
\begin{equation} \label{E:square function}
\|\scriptE \chi_{\Omega_\tau(K)} \scriptE \chi_{\Omega_{\tau'}(K')}\|_s^s \lesssim \sum_{k=0}^{K'} \sum_{j=0}^J \|\scriptE \chi_{\tau_k \cap \Omega(K)} \scriptE \chi_{\tau_j' \cap \Omega(K')}\|_s^s.
\end{equation}

We begin with the sum over those terms with $k=K'$.  By Cauchy--Schwarz and \eqref{E:OmegaK},
\begin{align*}
\sum_{j=0}^J \|\scriptE \chi_{\tau_{K'} \cap \Omega(K)} \scriptE \chi_{\tau_j' \cap \Omega(K')}\|_s^s
 &\lesssim \sum_{j=0}^J \|\scriptE \chi_{\tau_{K'} \cap \Omega(K)}\|_{2s}^s\|\scriptE\chi_{\tau_j' \cap \Omega(K')}\|_{2s}^s\\
  &\lesssim \sum_{j=0}^J |\tau_{K'}|^{\frac s{s'}}|\tau_j'|^{\frac s{s'}}.
\end{align*}
Because of the way the $\tau_j'$ were defined, we have at most two nonempty $\tau_j'$ with $j \leq J'$.  This, combined with the bound $|\tau_j'| \leq \min\{2^{-(j-J')},1\} |\tau'|$ gives $\sum_j |\tau_j'|^{\frac s{s'}} \lesssim |\tau'|^{\frac s{s'}}$ (despite the fact that $s < s'$).  Since $|\tau_{K'}| \sim 2^{-(K'-K)}|\tau|$, $|\tau| \sim |\Omega(K)|$, and $|\tau'| \sim |\Omega(K')|$, 
$$
\sum_{j=0}^J \|\scriptE \chi_{\tau_{K'} \cap \Omega(K)} \scriptE \chi_{\tau_j' \cap \Omega(K')}\|_s^s \lesssim  2^{-(K'-K) \frac s{s'}}|\Omega(K)|^{\frac s{s'}}|\Omega(K')|^{\frac s{s'}}.
$$

In the case $j=J$, a similar argument implies that
\begin{align*}
\sum_{k=0}^{K'} \|\scriptE \chi_{\tau_k \cap \Omega(K)}\scriptE \chi_{\tau_J' \cap \Omega(K')}\|_s^s &\lesssim 2^{-(J-J')\frac{s}{s'}}|\Omega(K)|^{\frac s{s'}} |\Omega(K')|^{\frac s{s'}}\\
&  \sim 2^{-(K'-K)\frac s{s'}}|\Omega(K)|^{\frac{2s}{s'}}.
\end{align*}

In the cases $k < K'$ and $j < J$, we have a gain, due to our bilinear extension estimate.  If $k < K'$ and $j < J$, $\tau_k$ is a (subset of four) tile(s) in $\scriptD_{J,\max\{k,K\}}$, $\tau_j$ is a (subset of four) tile(s) in $\scriptD_{\max\{j,J'\},K'}$, and these tiles are separated by a distance $2^{-k}$ in the vertical direction $2^{-j}$ in the horizontal direction.  These tiles are contained in separated tiles in $\scriptD_{j,k}$, so by the hypotheses of our theorem, for any $r < r_0$,
$$
\|\scriptE\chi_{\tau_k \cap \Omega(K)} \scriptE \chi_{\tau_j' \cap \Omega(K')}\|_s \lesssim  2^{-(j+k)(\frac2{s'} - \frac2r)} |\tau_k \cap \Omega(K)|^{\frac1r}|\tau_j' \cap \Omega(K')|^{\frac1r}.
$$
From our observation above that we have at most two values of $j$ (resp.\ $k$) in our sum with $j \leq J'$ (resp.\ $k \leq K$), our assumption that $r < s'$ gives
\begin{align*}
&\sum_{j=0}^{J} \sum_{k=0}^{K'} 2^{-(j+k)(\frac{2s}{s'} - \frac{2s}r)} |\tau_k \cap \Omega(K)|^{\frac sr}|\tau_j' \cap \Omega(K')|^{\frac sr}
 \leq \sum_{j=0}^J \sum_{k=0}^{K'} 2^{-(j+k)(\frac{2s}{s'} - \frac{2s}r)} |\tau_k|^{\frac sr}|\tau_j'|^{\frac sr}\\
&\qquad\qquad \lesssim 2^{-(J'+K)(\frac{2s}{s'} - \frac{2s}r)}|\tau|^{\frac sr}|\tau'|^{\frac sr}
\sim \delta^{-C}2^{-(J'+K)(\frac{2s}{s'} - \frac{2s}r)}|\Omega(K)|^{\frac sr}|\Omega(K')|^{\frac sr}\\
&\qquad\qquad \lesssim \delta^{-C} 2^{(J-J'+K'-K)(\frac sr-\frac s{s'})}|\Omega(K)|^{\frac s{s'}}|\Omega(K')|^{\frac s{s'}},
\end{align*}
which, by \eqref{E:square function}, is stronger than \eqref{E:log sep near tiles}
\end{proof}

We return to the proof of Lemma~\ref{L:Omega delta decouple}

Let $K_1,K_2,K_3,K_4 \in \scriptK$, not all equal.  Rearranging indices if needed, we may assume that $N_1 := K_1+J(K_1)$ is minimal among all $N_i$ and that $|K_1-K_4| \geq \tfrac12|K_i-K_j|$ for all $i,j$.  Thus $|\Omega(K_1)|$ is maximal.  By H\"older's inequality and Lemma~\ref{L:separated tiles},
$$
\|\prod_{i=1}^4 \scriptE \chi_{\Omega_\delta(K_i)}\|_{\frac s2} \lesssim 2^{-c_0|K_1-K_4|} |\Omega(K_1)|^{\frac4{s'}}.
$$
Therefore
$$
{\sum}' \|\prod_{i=1}^4 \scriptE \chi_{\Omega_\delta(K_i)}\|_{\frac s2} \lesssim \sum_{K_1 \in \scriptK} \sum_{K_1 \neq K_4 \in \scriptK} |K_4-K_1|^2 2^{-c_0|K_4-K_1|}|\Omega(K_1)|^{\frac{2s}{s'}}.
$$
Since $2s > s'$ and $\scriptK$ is $A \log \delta^{-1}$-separated for some very large $A$, our error term is bounded by $\delta^C |\Omega|^{\frac{2s}{s'}}$.
\end{proof}

\begin{proof}[Proof of Theorem~\ref{T:main}]  We decompose $\Omega$ by fiber length, and decompose the fiber lengths according to the exactness of Vargas's estimate:  
\begin{align*}
\|\scriptE \chi_\Omega\|_{2s} &\leq \sum_{0 < \eps \lesssim 1}\sum_{0 < \delta \leq \eps} \|\sum_{K \in \scriptK(\eps)} \scriptE \chi_{\Omega_\delta(K)}\|_{2s} \\
&\lesssim \sum_{0 < \eps \lesssim 1}\sum_{0 < \delta \leq \eps} \bigl[\bigl(\log \delta^{-1} \sum_{K \in \scriptK(\eps)} \|\scriptE \chi_{\Omega_\delta(K)}\|_{2s}^{2s}\bigr)^{\frac1{2s}}+\delta|\Omega|^{\frac1{s'}}\bigr]\\
&\lesssim \sum_{0 < \eps \lesssim 1} \sum_{0 < \delta \leq \eps} \bigl[\bigl((\log\delta^{-1}\delta)^{2s} \sum_{K \in \scriptK(\eps)} |\Omega(K)|^{\frac{2s}{s'}}\bigr)^{\frac1{2s}} + \eps|\Omega|^{\frac1{s'}}\bigr]\\
&\lesssim \sum_{0 < \eps \lesssim 1} \sum_{0 < \delta \leq \eps} \log \delta^{-1}\delta |\Omega|^{\frac1{s'}} \lesssim |\Omega|^{\frac1{s'}},
\end{align*}
where, for the second to last inequality we are using the fact that $2s > s'$ and the triangle inequality for $\ell^{\frac{2s}{s'}}$ to sum the volumes of the $\Omega(K)$.  
\end{proof}



\end{document}